\documentclass[a4paper,reqno, 12pt]{amsart}
\usepackage[T1]{fontenc}
\usepackage{amsmath}
\usepackage{amsthm}
\usepackage{amssymb}
\usepackage{mathrsfs}
\usepackage{graphicx}
\usepackage{hyperref}
\hypersetup{colorlinks=true,linkcolor=black,citecolor=black,filecolor=black,urlcolor=black}
\usepackage[utf8]{inputenc}
\usepackage{tikz}

\newcommand{\R}{\mathbb{R}}
\DeclareMathOperator{\stp}{P_{\mathrm{sp.trees}}}
\DeclareMathOperator{\nesubp}{P^\star_{\mathrm{sub}}}
\DeclareMathOperator{\subp}{P_{\mathrm{sub}}}

\DeclareMathOperator{\conv}{conv}
\DeclareMathOperator{\xc}{xc}

\newtheorem{prop}{Proposition}
\newtheorem{conj}{Conjecture}

\newtheorem{lem}[prop]{Lemma}

\newtheorem{thm}[prop]{Theorem}

\newtheoremstyle{question}{}{}{\color{blue}}{}{\color{blue}\bfseries}{}{ }{}
\theoremstyle{question}

\newtheoremstyle{openquestion}{}{}{\color{red}}{}{\color{red}\bfseries}{}{ }{}
\theoremstyle{openquestion}

\begin{document}
\title[Spanning Tree Polytope of Bounded-genus Graphs]{Smaller Extended Formulations for the Spanning Tree Polytope of Bounded-genus Graphs}

\makeatletter
\let\old@setaddresses\@setaddresses
\def\@setaddresses{\bigskip\bgroup\parindent 0pt\let\scshape\relax\old@setaddresses\egroup}
\makeatother

\author[S.~Fiorini]{Samuel Fiorini}
\author[T.~Huynh]{Tony Huynh}
\address[S.~Fiorini, T.~Huynh]{Mathematics Department \\
  Universit\'e Libre de Bruxelles\\
  Brussels\\
  Belgium}
\email{sfiorini@ulb.ac.be, tony.bourbaki@gmail.com}

\author[G.~Joret]{Gwena\"{e}l Joret}
\address[G.~Joret]{Computer Science Department \\
  Universit\'e Libre de Bruxelles\\
  Brussels\\
  Belgium}
\email{gjoret@ulb.ac.be}

\author[K.~Pashkovich]{Kanstantsin Pashkovich}
\address[K.~Pashkovich]{
Department of Combinatorics and Optimization \\
University of Waterloo \\
Waterloo \\
Canada
}
\email{kanstantsin.pashkovich@gmail.com}

\maketitle

\sloppy

\begin{abstract}
We give an $O(g^{1/2} n^{3/2} + g^{3/2} n^{1/2})$-size extended formulation for the spanning tree polytope of an $n$-vertex graph embedded in a surface of genus $g$, improving on the known $O(n^2 + g n)$-size extended formulations following from Wong~\cite{Wong80} and Martin~\cite{Martin91}.
\end{abstract}

\section{Introduction}

An \emph{extended formulation} of a (convex) polytope $P \subseteq \R^d$ is a linear system $Ax + By \leqslant b,\ Cx + Dy = c$ in variables $x \in \R^d$ and $y \in \R^k$ that provides a description of $P$ in the sense that
$$
P = \{x \in \R^d \mid \exists y \in \R^k : Ax + By \leqslant b,\ Cx + Dy = c\}\,.
$$
The \emph{size} of an extended formulation is defined as its number of inequalities. The \emph{extension complexity} $\xc(P)$ is the minimum size of an extended formulation of $P$. Notice that equalities are not accounted for in the size of an extended formulation. In fact, we may equivalently define the extension complexity of $P$ as the minimum number of facets of a polytope that affinely projects to $P$.

Let $G = (V,E)$ be a connected (simple, finite, undirected) graph. The \emph{spanning tree polytope} of $G$ is the convex hull of the $0/1$-vectors in $\R^E$  that are the characteristic vector of some spanning tree of $G$. We denote this polytope as $\stp(G)$, and use the notation
$$
\stp(G) = \conv \{\chi^T \in \{0,1\}^E \mid T \subseteq E,\ T \text{ spanning tree of } G\}\,.
$$

The following result gives the best known upper bound on the extension complexity of the spanning tree polytope for general graphs, and is due to Wong~\cite{Wong80} and Martin~\cite{Martin91}.

\begin{thm} \label{thm:Martin_stp}
For every connected graph $G = (V,E)$,  $\xc(\stp(G)) = O(|V| \cdot |E|)$.
\end{thm}

For planar graphs, a linear bound was proved by Williams~\cite{Williams2002}.

\begin{thm} \label{thm:Williams}
For every connected planar graph $G = (V,E)$,  $\xc(\stp(G)) = O(|V|)$.
\end{thm}

Let $\mathcal{S}$ be a surface.  By the classification theorem for surfaces, $\mathcal{S}$ is homeomorphic to a sphere with $g$ handles, or a sphere with $g$ crosscaps, for some $g$.  We call $g$ the \emph{genus} of $\mathcal{S}$.  Our main result is an improvement of Theorem \ref{thm:Martin_stp} for graphs embedded in a surface of genus $g$.  

\begin{thm} \label{thm:main}
For every connected graph $G = (V,E)$ embedded in a surface of genus $g$,  
$\xc(\stp(G)) = O(g^{1/2} |V|^{3/2} + g^{3/2} |V|^{1/2})$. 
In particular, $\xc(\stp(G)) = O(|V|^{3/2})$ if $g$ is fixed. 
\end{thm}

This gives an improvement over Theorem~\ref{thm:Martin_stp} for all fixed $g$. For instance, for toroidal graphs we obtain a $O(|V|^{3/2})$-size extended formulation, while the previously known extended formulations are of size $\Omega(|V|^2)$.


For other polytopes, smaller extended formulations have also been obtained when restricting to graphs of bounded genus.  For example, Gerards~\cite{Gerards91} proved that the perfect matching polytope has a polynomial-size extended formulation for graphs embedded in a fixed genus surface.  This is in stark contrast to the situation for general graphs: Rothvoß~\cite{Rothvoss14} showed that the perfect matching polytopes of complete graphs have exponential extension complexity. 

Going back to the spanning tree polytope, we conjecture that the bound in Theorem~\ref{thm:main} can be improved to match the corresponding bound for planar graphs.  

\begin{conj}
If $G=(V,E)$ is a connected graph embedded in a fixed  surface, then $\xc(\stp(G)) = O(|V|)$.
\end{conj}

Indeed, the same bound may even hold more generally for proper minor-closed families of graphs.  

\begin{conj}
If $\mathcal{C}$ is a proper minor-closed family of graphs and $G=(V,E)$ is a connected graph in $\mathcal{C}$, then $\xc(\stp(G)) = O(|V|)$.
\end{conj}

We remark that this conjecture is known to hold if the graphs in $\mathcal{C}$ have bounded treewidth~\cite{KolmanKT16}. 
To provide some additional support for the conjecture, we observe that it is also true when the graphs in $\mathcal{C}$ are $k$-apex for some fixed $k$. 
Recall that a graph $G=(V,E)$ is \emph{$k$-apex} if there is a set $X \subseteq V$ with $|X| \leqslant k$ such that $G-X$ is planar. It is easily checked that the set of $k$-apex graphs is a proper minor-closed family of graphs.

\begin{thm} \label{thm:kapex}
Let $G = (V,E)$ be a connected $k$-apex graph.  Then $\xc(\stp(G)) = O(k \cdot |E|) = O(k^2 \cdot |V|)$.
\end{thm}

\section{The Proofs}

In this section we prove Theorems \ref{thm:main} and \ref{thm:kapex}. We first gather the necessary ingredients.

As before, let $G = (V,E)$ be a connected graph. The \emph{subgraph polytope} of $G$ is defined as $\subp(G) = \conv \{(\chi^S,\chi^F) \in \{0,1\}^{V} \times \{0,1\}^{E} \mid S \subseteq V,\ F \subseteq E(S)\}$, where $E(S)$ denotes the set of edges of $G$ with both endpoints in $S$. 
It is easy to verify using total unimodularity that $\subp(G) = \{(x,y) \in \R^V \times \R^E \mid \forall v,w \in V \text{ with } vw \in E : 0 \leqslant y_{vw} \leqslant x_v \leqslant 1\}$. Hence, the subgraph polytope has at most $3 |E| + |V|$ facets, and in particular $\xc(\subp(G)) = O(|E|)$.

We will mostly be interested in the variant of the subgraph polytope known as the \emph{non-empty subgraph polytope}, defined as $\nesubp(G) = \conv \{(\chi^S,\chi^F) \in \{0,1\}^{V} \times \{0,1\}^{E} \mid \emptyset \subsetneq S \subseteq V,\ F \subseteq E(S)\}$. Notice that $\nesubp(G)$ is nothing else than the convex hull of the vertices of $\subp(G)$ distinct from the origin $(\mathbf{0}^V,\mathbf{0}^E)$. 

The non-empty subgraph polytope turns out to be tightly connected to the spanning tree polytope: Conforti, Kaibel, Walter and Weltge \cite{ConfortiKWW15} proved that the extension complexities of the two polytopes are essentially equal.

\begin{thm}
\label{thm:equiv}
For every connected graph $G = (V,E)$,  $\xc(\stp(G)) = \xc(\nesubp(G)) + \Theta(|E|)$.
\end{thm}

In particular, it follows from this and Theorem~\ref{thm:Williams} that $\xc(\nesubp(G)) = O(|V|)$ for every connected planar graph $G = (V,E)$. 

Balas' union of polytopes \cite{Balas88} is a basic tool to construct extended formulations. It provides an upper bound on the extension complexity of the convex hull of a union of polytopes.

\begin{thm} \label{thm:Balas}
Let $P_1$, \ldots, $P_k$ be non-empty polytopes in $\R^d$, and let $P = \conv\left(\bigcup_{i=1}^k P_i \right)$. Then $\xc(P) \leqslant \sum_{i=1}^k \max \{1, \xc(P_i)\}$.
\end{thm}

The next observation follows easily from Balas' union of polytopes.

\begin{lem} \label{lem:deletion}
Let $G = (V,E)$ be a connected graph and $X \subseteq V$ be a set of vertices. Then
$$
\nesubp(G) = \conv \left( \nesubp(G-X) \cup \bigcup_{v \in X} (\subp(G) \cap \{(x,y) \in \R^V \times \R^E \mid x_v = 1\}) \right)\,,
$$
thus $\xc(\nesubp(G)) \leqslant \xc(\nesubp(G-X)) + O(|X| \cdot |E|)$.
\end{lem}

\begin{proof}
We may assume that $X$ is a proper, non-empty subset of $V$, since otherwise the result holds. From Theorem~\ref{thm:Balas},
\begin{align*}
\xc(\nesubp(G)) 
&\leqslant \xc(\nesubp(G-X)) + \sum_{v \in X} \xc(\subp(G))\\
&= \xc(\nesubp(G-X)) + O(|X| \cdot |E|)\,. 
\end{align*}
(Remark: If $|X|=|V|-1$ then $\nesubp(G-X)$ is empty and is thus not part of the list of polytopes we apply Theorem~\ref{thm:Balas} on, as expected.)
\end{proof}

We are now ready to prove Theorem~\ref{thm:kapex}.

\begin{proof}[Proof of Theorem~\ref{thm:kapex}]
Let $G=(V,E)$ be a connected $k$-apex graph, and let $X \subseteq V$ be any set of at most $k \geqslant 1$ vertices whose deletion from $G$ gives a planar graph. By Theorem~\ref{thm:equiv}, Lemma~\ref{lem:deletion} and Theorem~\ref{thm:Williams},
\begin{align*}
\xc(\stp(G)) 
&\leqslant \xc(\nesubp(G)) + O(|E|)\\
&\leqslant \underbrace{\xc(\nesubp(G-X))}_{= O(|V|) = O(|E|)} + O(\underbrace{|X|}_{\leqslant k} \cdot |E|) = O(k \cdot |E|)\,.
\end{align*}
Notice that $|E| \leqslant k \cdot (|V|-1) + 3 (|V|-k) - 6 = O(k \cdot |V|)$, thus $O(k \cdot |E|) = O(k^2 \cdot |V|)$.
\end{proof}

For Theorem~\ref{thm:main}, we need one additional result of Djidjev and Venkatesan \cite{DV95}. The same result for orientable surfaces was obtained earlier by Hutchinson and Miller~\cite{HutchinsonM86}.

\begin{thm} \label{thm:planarization}
For every graph $G = (V,E)$ embedded in a surface of genus $g$, there exists a set $X$ of $O(\sqrt{g|V|})$ vertices such that $G - X$ is planar.
\end{thm}

Finally, we prove our main result, Theorem~\ref{thm:main}.

\begin{proof}[Proof of Theorem~\ref{thm:main}]
Let $G=(V,E)$ be a connected graph embedded in a surface of genus $g$. The result follows by combining Theorem \ref{thm:equiv}, Lemma~\ref{lem:deletion}, Theorem \ref{thm:Williams}, Theorem \ref{thm:planarization}, and the upper bound $|E| = O(|V| + g)$ (by Euler's formula).

More explicitly, letting $X \subseteq V$ be as in Theorem~\ref{thm:planarization}, 
\begin{align*}
\xc(\stp(G)) 
&\leqslant \xc(\nesubp(G)) + O(|E|)\\
&\leqslant \xc(\nesubp(G-X))+ O(|X| \cdot |E|)\\
&= O(|V|) + O(g^{1/2} |V|^{1/2} \cdot (|V|+g))\\
&= O(g^{1/2} |V|^{3/2} + g^{3/2} |V|^{1/2})\,. \qedhere
\end{align*}
\end{proof}

\section*{Acknowledgments}

We thank William Cook for asking about the extension complexity of the spanning tree polytope of graphs on the torus, which prompted this work. We thank also Jean Cardinal and Hans Raj Tiwary for interesting discussions on the topic, and the two referees for rightfully asking us to shorten the paper.

S.~Fiorini and T.~Huynh are supported by ERC grant \emph{FOREFRONT} (grant agreement no.\ 615640) funded by the European Research Council under the EU's 7th Framework Programme (FP7/2007-2013). S.~Fiorini also acknowledges support from ARC grant AUWB-2012-12/17-ULB2 \emph{COPHYMA} funded by the French community of Belgium. 
S.~Fiorini and K.~Pashkovich are grateful for the support of the  Hausdorff Institute for Mathematics in Bonn during the trimester program \emph{Combinatorial Optimization}. 
G.~Joret acknowledges support from an ARC grant from the Wallonia-Brussels Federation of Belgium. 

\bibliography{notes}{}

\begin{thebibliography}{10}

\bibitem{Balas88}
Egon Balas.
\newblock On the convex hull of the union of certain polyhedra.
\newblock {\em Oper. Res. Lett.}, 7(6):279--283, 1988.

\bibitem{ConfortiKWW15}
Michele Conforti, Volker Kaibel, Matthias Walter, and Stefan Weltge.
\newblock Subgraph polytopes and independence polytopes of count matroids.
\newblock {\em Operations Research Letters}, 43(5):457 -- 460, Sep 2015.

\bibitem{DV95}
Hristo~N. Djidjev and Shankar~M. Venkatesan.
\newblock Planarization of graphs embedded on surfaces.
\newblock In {\em Graph-theoretic concepts in computer science ({A}achen,
  1995)}, volume 1017 of {\em Lecture Notes in Comput. Sci.}, pages 62--72.
  Springer, Berlin, 1995.

\bibitem{Gerards91}
A.~M.~H. Gerards.
\newblock Compact systems for {$T$}-join and perfect matching polyhedra of
  graphs with bounded genus.
\newblock {\em Oper. Res. Lett.}, 10(7):377--382, 1991.

\bibitem{HutchinsonM86}
Joan~P. Hutchinson and Gary~L. Miller.
\newblock On deleting vertices to make a graph of positive genus planar.
\newblock In {\em Discrete algorithms and complexity ({K}yoto, 1986)},
  volume~15 of {\em Perspect. Comput.}, pages 81--98. Academic Press, Boston,
  MA, 1987.

\bibitem{KolmanKT16}
Petr Kolman, Martin Kouteck{\'y}, and Hans~Raj Tiwary.
\newblock {Extension Complexity, MSO Logic, and Treewidth }.
\newblock In Rasmus Pagh, editor, {\em 15th Scandinavian Symposium and
  Workshops on Algorithm Theory (SWAT 2016)}, volume~53 of {\em Leibniz
  International Proceedings in Informatics (LIPIcs)}, pages 18:1--18:14,
  Dagstuhl, Germany, 2016. Schloss Dagstuhl--Leibniz-Zentrum fuer Informatik.

\bibitem{Martin91}
R.~Kipp Martin.
\newblock Using separation algorithms to generate mixed integer model
  reformulations.
\newblock {\em Oper. Res. Lett.}, 10(3):119--128, 1991.

\bibitem{Rothvoss14}
Thomas Rothvoss.
\newblock The matching polytope has exponential extension complexity.
\newblock In {\em S{TOC}'14---{P}roceedings of the 2014 {ACM} {S}ymposium on
  {T}heory of {C}omputing}, pages 263--272. ACM, New York, 2014.

\bibitem{Williams2002}
Justin~C. Williams.
\newblock A linear-size zero-one programming model for the minimum spanning
  tree problem in planar graphs.
\newblock {\em Networks}, 39(1):53--60, 2002.

\bibitem{Wong80}
R.T. Wong.
\newblock Integer programming formulations of the traveling salesman problem.
\newblock In {\em Proceedings of 1980 IEEE International Conference on Circuits
  and Computers}, pages 149--152, 1980.

\end{thebibliography}
\bibliographystyle{plain}

\end{document}